\documentclass[11pt]{article}
\usepackage{amssymb}
\usepackage{mathrsfs}
\usepackage{amsthm}
\usepackage{amsmath}
\usepackage{graphicx}
\usepackage{tikz}
\usepackage{lmodern}
\usepackage{enumerate}

\usepackage{latexsym,amsmath,amssymb,amsfonts,epsfig,graphicx,cite,psfrag}
\usepackage{eepic,color,colordvi,amscd}

\newtheorem{thm}{Theorem}[section]

\newtheorem{lem}[thm]{Lemma}

\newtheorem{constr}[thm]{Construction}

\newtheorem{conj}[thm]{Conjecture}
\newtheorem{defi}[thm]{Definition}

\newcommand*{\1}{1}
\newcommand*{\2}{2}
\newcommand*{\3}{3}
\newcommand*{\4}{4}
\newcommand*{\5}{5}
\newcommand*{\6}{6}

\DeclareMathOperator{\RS}{RSz}
\DeclareMathOperator{\Di}{Disj}
\DeclareMathOperator{\di}{disj}

\newcommand{\Bin}[1]{\text{Bin}\left[ #1 \right]}

\begin{document}

\title{Diameter critical graphs}
\author{
Po-Shen Loh \thanks{Department of Mathematical Sciences, Carnegie Mellon
University, Pittsburgh, PA 15213. E-mail: ploh@cmu.edu. Research supported
in part by NSF grant DMS-1201380 and by a USA-Israel BSF Grant.}
\and
Jie Ma\thanks{Department of Mathematical Sciences, Carnegie Mellon
University, Pittsburgh, PA 15213. Email: jiemath@andrew.cmu.edu.}
}

\date{}

\maketitle

\begin{abstract}
  A graph is called diameter-$k$-critical if its diameter is $k$, and the
  removal of any edge strictly increases the diameter.  In this paper, we
  prove several results related to a conjecture often attributed to Murty
  and Simon, regarding the maximum number of edges that any
  diameter-$k$-critical graph can have.  In particular, we disprove a
  longstanding conjecture of Caccetta and H\"aggkvist (that in every
  diameter-2-critical graph, the average edge-degree is at most the number
  of vertices), which promised to completely solve the extremal problem for
  diameter-2-critical graphs.
  
  On the other hand, we prove that the same claim holds for all higher
  diameters, and is asymptotically tight, resolving the average edge-degree
  question in all cases except diameter-2.  We also apply our techniques to
  prove several bounds for the original extremal question, including the
  correct asymptotic bound for diameter-$k$-critical graphs, and an upper
  bound of $(\frac{1}{6} + o(1))n^2$ for the number of edges in a
  diameter-3-critical graph.
\end{abstract}

\section{Introduction}

An {\em $(x,y)$-path} is a path with endpoints $x$ and $y$, and its {\em
length}\/ is its number of edges.  We denote by $d_G(x,y)$ the smallest
length of an $(x,y)$-path in a graph $G$, where we often drop the subscript
if the graph $G$ is clear from context.  The {\em diameter} of $G$ is the
maximum of $d_G(x,y)$ over all pairs $\{x,y\}$.  A graph $G$ is said to be
\emph{diameter-critical} if for every edge $e \in G$, the deletion of $e$
produces a graph $G - e$ with higher diameter.

The area of diameter-criticality is one of the oldest subjects of study in
extremal graph theory, starting from papers of Erd\H{o}s-R\'enyi \cite{ER},
Erd\H{o}s-R\'enyi-S\'os \cite{ERS}, Murty-Vijayan \cite{MV}, Murty
\cite{Mu1, Mu2, Mu3}, and Ore \cite{Ore} from the 1960's.  Many problems in
this domain were investigated, such as that of minimizing the number of
edges subject to diameter and maximum-degree conditions (see, e.g.,
Erd\H{o}s-R\'enyi \cite{ER}, Bollob\'as \cite{Bo1, Bo2},
Bollob\'as-Eldridge \cite{BEl}, Bollob\'as-Erd\H{o}s \cite{BEr}),
controlling post-deletion diameter (Chung \cite{Ch}), and
vertex-criticality (Caccetta \cite{Ca}, Erd\H{o}s-Howorka \cite{EH},
Huang-Yeo \cite{HY}, Chen-F\"uredi \cite{CF}), to name just a few.

The natural extremal problem of maximizing the number of edges  (or
equivalently, the average degree) in a diameter-critical graph also
received substantial attention.  Our work is inspired by the following
long-standing conjecture of Ore \cite{Ore}, Plesn\'ik \cite{Pl}, Murty and
Simon (see in \cite{CH}).  A \emph{diameter-$k$-critical} graph is a
diameter-critical graph of diameter $k$.

\begin{conj}
  \label{conj:murty-simon}
  For each $n$, the unique diameter-2-critical graph which maximizes the
  number of edges is the complete bipartite graph $K_{\lfloor n/2 \rfloor,
  \lceil n/2 \rceil}$.
\end{conj}

Successively stronger estimates were proved in the 1980's by Plesn\'ik
\cite{Pl}, Cacetta-H\"aggkvist \cite{CH}, and Fan \cite{Fan}, culminating
in a breakthrough by F\"uredi \cite{Fu}, who used a clever application of
the Ruzsa-Szemer\'edi (6,3) theorem to prove the exact (non-asymptotic)
result for large $n$.  As current quantitative bounds on the (6,3) theorem
are of tower-type, the constraint on $n$ is quite intense, and there is
interest in finding an approach which is free of Regularity-type
ingredients.  For example, the recent survey by Haynes, Henning, van der
Merwe, and Yeo \cite{HHMY-survey} discusses recent work following a
different approach based upon total domination, but we do not pursue that
direction in this paper.

One hope for a Regularity-free method was proposed at around the origin of
the early investigation.  In their original 1979 paper, Caccetta and
H\"aggkvist posed a very elegant stronger conjecture for a related problem,
which would establish the extremal number of edges in diameter-2-critical
graphs for all $n$.  For an edge $e$, let its \emph{edge-degree} $d(e)$ be
the sum of the degrees of its endpoints, and let $\overline{d(e)}$ be the
average edge-degree over all edges, so that in terms of the total number of
edges $m$,
\begin{displaymath}
  \overline{d(e)} 
  = 
  \frac{1}{m} \sum_{uv \in E(G)} (d_u + d_v)
  =
  \frac{1}{m} \sum_{v \in V(G)} d_v^2 .
\end{displaymath}
Just as the Ore-Plesn\'ik-Murty-Simon problem sought to maximize the
average vertex-degree over diameter-critical graphs, one can also ask to
maximize the average edge-degree.

\begin{conj}
  \label{conj:CHd2}
  (Caccetta-H\"aggkvist \cite{CH}; also see \cite{Fu}) For any
  diameter-2-critical graph, the average edge-degree is at most the number
  of vertices.
\end{conj}

In terms of the numbers of vertices and edges ($n$ and $m$), the conclusion
of this conjecture is equivalent to:
\begin{displaymath}
  \sum_v d_v^2 \leq nm .
\end{displaymath}

Given the Caccetta-H\"aggkvist conjecture, Conjecture
\ref{conj:murty-simon} then follows as an immediate consequence of
convexity: $\sum d_v^2$ is at least $n$ times the square of the average
degree, and so Conjecture \ref{conj:CHd2} implies that $nm \geq n
(2m/n)^2$, giving $m \leq n^2/4$.  In \cite{CH}, Caccetta and H\"aggkvist
proved the constant-factor approximation $\sum d_v^2 \leq \frac{6}{5} nm$,
but there was no improvement for over three decades.  Our first result
indicates why: the Caccetta-H\"aggkvist conjecture is false.  We
demonstrate this by constructing a rich family of diameter-2-critical
graphs, which may be of independent interest, as one challenge in the study
of diameter-critical graphs is to find broad families of examples.  (See
our Constructions \ref{constr:d2-bip} and \ref{constr:d2-trip}).

\begin{thm}
  \label{thm:CHd2-false}
  There is an infinite family of diameter-2-critical graphs for which
  \begin{displaymath}
   \overline{d(e)} \geq \left( \frac{10}{9} - o(1) \right) n,
  \end{displaymath}
  where $o(1)$ tends to $0$ as $n$ tends to infinity.
\end{thm}

Since our construction opens a constant factor gap, this shows that the
Caccetta-H\"aggkvist conjecture cannot be used to resolve the original
problem.  We were intrigued to study the best value of the constant
multiplier on the Caccetta-H\"aggkvist conjecture, as it is a natural question in its own
right, and proved that the $\frac{6}{5}$ factor from \cite{CH} was not
sharp either.

\begin{thm}
  \label{thm:CHd2-improve}
  There are absolute constants $c$ and $N$ such that for every
  diameter-2-critical graph with at least $N$ vertices, the average edge-degree is at most $(
  \frac{6}{5} - c)$ times the number of vertices.
\end{thm}

On the other hand, it turns out that for all $k \geq 3$, the
diameter-$k$-critical analogue of the Caccetta-H\"aggkvist conjecture holds
precisely, and is tight.

\begin{thm}
  \label{thm:CH-diamk}
  For every diameter-critical graph with diameter at least 3, the average
  edge-degree is at most the number of vertices.  This is asymptotically
  tight: for each fixed $k \ge 3$, there is an infinite family of graphs
  for which the average edge-degree is at least $(1-o(1))$ times the number
  of vertices.
\end{thm}

As an immediate consequence, the earlier convexity argument proves that
diameter-critical graphs with diameter at least 3 have at most $n^2/4$
edges, for all $n$. Regarding the maximum number of edges in diameter-$k$-critical graphs
for $k \ge 3$, however, it was conjectured by
Krishnamoorthy and Nandakumar \cite{KN} (who disproved an earlier
conjecture of Caccetta and H\"aggkvist on this question) that a particular
instance of the following construction is optimal (see Lemma \ref{lem:constr:dK} for a proof of diameter-criticality).

\begin{constr}
  \label{constr:dK}
  Let $a$, $b$, and $c$ be positive integers.  Create a partition $V_0 \cup
  V_1 \cup \cdots \cup V_k$ such that $|V_0| = a$,
  $|V_1|=|V_2|=\ldots=|V_{k-1}|=b$ and $|V_k|=c$.  Introduce edges by
  placing complete bipartite graphs between $V_0$ and $V_1$, and between
  $V_{k-1}$ and $V_k$, and placing $b$ vertex-disjoint paths of length
  $k-2$ from $V_1$ to $V_{k-1}$, each with exactly one vertex in $V_i$ for
  every $1\le i\le k-1$.
\end{constr}

Krishnamoorthy and Nandakumar \cite{KN} observed that choosing $a=1$, $b
\approx \frac{n}{2(k-1)}$, and $c = n-a-b(k-1)$ optimized the number of
edges in this construction, yielding a total of $\frac{n^2}{4(k-1)}+o(n^2)$
edges.  Our next result establishes a bound which deviates by an constant
factor.

\begin{thm}
  \label{thm:diamk-edges}
  Every diameter-$k$-critical graph on $n$ vertices has at most
  $\frac{3n^2}{k}$ edges.
\end{thm}

Finally, we investigated the case $k=3$ in greater depth, as it is the
first (asymptotically) unresolved diameter.  The above construction
produces a diameter-3-critical graph with $(\frac{1}{8} + o(1)) n^2$ edges.
We point out that there is a significantly different graph with the same
asymptotic edge count: a clique $A:=K_{n/2}$ together with a perfect
matching (with $n/2$ edges) between $A$ and its complement $A^c$.  On the
other hand, as observed above, our Theorem \ref{thm:CH-diamk} immediately
gives $n^2/4$ as an upper bound.  We improve this to an intermediate value,
and note that our proof applies in diameters greater than 3 as well (see
remark at end of Section \ref{sec:diam3}).

\begin{thm}
  \label{thm:diam3}
  Every diameter-3-critical graph on $n$ vertices has at most
  $\frac{n^2}{6} + o(n^2)$ edges.
\end{thm}

The rest of this paper is organized as follows.  The next section contains
constructions of families of diameter-2-critical graphs, which contain
counterexamples to the Caccetta-H\"aggkvist conjecture.  Section 3 resolves
the Caccetta-H\"aggkvist conjecture for all diameters $k \geq 3$.  In
Section 4, we improve the upper bound on the diameter-2
Caccetta-H\"aggkvist result.  Section 5 establishes an upper bound on the
edge count for all diameters, which deviates by a constant factor (12) from
all best known constructions.  Finally, Section 6 proves a substantially
stronger upper bound for the diameter-3 case.

Throughout our proofs, we will encounter and manipulate many paths.  We
will use $ab$ to denote an edge, and $abc$, $abcd$, etc.  to denote paths.
If $u$ and $v$ are vertices of a path $P$, we will write $uPv$ to denote
the subpath of $P$ with $u$ and $v$ as endpoints.

\section{Diameter-2-critical constructions}

In this section, we prove Theorem \ref{thm:CHd2-false}, by constructing a
very rich family of diameter-2-critical graphs.  Indeed, a major challenge
in studying diameter-2-critical graphs is the lack of understanding of the
menagerie of examples of such graphs.  We proceed with a series of two
constructions which give rise to a broad variety of examples.

\begin{constr}
  \label{constr:d2-bip}
  Let $G$ be an arbitrary $n$-vertex graph for which both $G$ and its
  complement have diameter at most 2.  Create a new graph $G'$ by taking
  two disjoint copies $A$ and $B$ of the set $V(G)$, placing an induced copy of $G$ in $A$,
  placing an induced copy of the complement of $G$ in $B$, and placing a perfect matching between $A$ and $B$
  such that each edge in this matching joins a vertex in $A$ to its copy in $B$. Then $G'$ is
  diameter-2-critical.
\end{constr}

\begin{proof}
  Select vertices $x \in A$ and $y \in B$ so that $xy$ is not an edge of
  the perfect matching.  Let $x' \in B$ and $y' \in A$ be the respective
  partners of $x$ and $y$ according to the perfect matching.  Then, exactly
  one of $x' y$ and $x y'$ is in $G'$, while $xy$ is not, and thus
  distance between $x$ and $y$ is exactly 2.  At the same time, vertices in
  the same part are at distance at most 2, since both $G$ and its
  complement have diameter at most 2.  Therefore, $G'$ has diameter equal
  to 2.

  Also, it is clear that upon deleting any matching edge $x x'$ (where $x
  \in A$ and $x' \in B$), the distance between $x$ and $x'$ rises to at
  least 3.  Deleting any edge $xy$ in $A$ increases the distance between
  $x$ and $y'$ above 2, where $y' \in B$ is the partner of $y \in A$,
  and so we conclude that $G'$ is indeed diameter-2-critical.
\end{proof}

At this point, although we can generate a wide variety of
diameter-2-critical graphs, they do not yet bring $\sum d_v^2$ above $nm$.
However, if one uses a very sparse graph $G$ (with $o(n^2)$ edges) as the
generator, one finds that $\sum d_v^2$ is asymptotically $nm$, while being
very different from the balanced complete bipartite graph that also
achieves that bound.  As the balanced complete bipartite graph was quite a
stable optimum, this second point creates the possibility for us to
destabilize it.  We exploit this by augmenting the construction with a
third part.

\begin{constr}
  \label{constr:d2-trip}
  Let $r$ be an arbitrary natural number, and let $G$ be an $n$-vertex
  graph for which both $G$ and its complement have diameter at most 2.
  Create $G'$ from $G$ as in Construction \ref{constr:d2-bip}, and add a
  third disjoint set $C$ of $r$ vertices, with a complete bipartite graph
  between $B$ and $C$.  Then, the resulting $(2n+r)$-vertex graph $G''$ is
  diameter-2-critical.
\end{constr}

\begin{proof}
  We build upon our existing knowledge about $G'$.  It is clear that each
  vertex of $C$ is at distance at most 2 from every other vertex, and so
  $G''$ has diameter 2 as well.  Also, the deletion of any edge $yz$ with
  $y \in B$ and $z \in C$ would put $z$ at distance greater than 2 from $y'
  \in A$ (the matching partner of $y \in B$).  Therefore, $G''$ is
  diameter-2-critical, as claimed.
\end{proof}

We build our counterexample to the Caccetta-H\"aggkvist conjecture by
selecting suitable $n$ and $r$, and using a sparse diameter-2 random graph
$G$ whose complement also has diameter 2.  We saw that a property holds
\emph{asymptotically almost surely}, or a.a.s., if its probability tends to
1 as $n \rightarrow \infty$.  The random graph $G_{n,p}$ is constructed by
starting with $n$ vertices, and taking each of the $\binom{n}{2}$ potential
edges independently with probability $p$.

\begin{lem}
  Let $2 \sqrt{\frac{\log n}{n}} \leq p \leq 1 - 2 \sqrt{\frac{\log n}{n}}$.
  Then, a.a.s., $G_{n,p}$ and its complement both have diameter at most 2,
  and all vertex degrees are at most $2np$.
  \label{lem:sparse-diam-2-and-complement}
\end{lem}

\begin{proof}
  For any fixed pair of vertices, the probability that they have no common
  neighbors is exactly $(1-p^2)^{n-2}$.  A union bound over all pairs of
  vertices produces a total failure probability of at most $n^2 e^{-p^2
  (n-2)}$, which is clearly $o(1)$ because $p \geq 2 \sqrt{\frac{\log
  n}{n}}$.  Similarly, since $(1-p) \geq 2 \sqrt{\frac{\log n}{n}}$, we see
  that a.a.s., both $G_{n,p}$ and its complement have diameter at most 2.
  For the degrees, since a given vertex's degree is distributed as
  $\Bin{n-1, p}$, the Chernoff inequality implies that the probability it
  exceeds $2np$ is at most $e^{-\Theta(np)}$, and another union bound over
  all vertices implies the result.
\end{proof}

We are now ready to prove our first result, showing that there exist graphs
which have $\sum d_v^2$ significantly greater than the product of
their numbers of vertices and edges.

\begin{proof}[Proof of Theorem \ref{thm:CHd2-false}]
  Use $p = 2 \sqrt{\frac{\log n}{n}}$ to create a random $n$-vertex graph
  $G$ which satisfies the properties in Lemma
  \ref{lem:sparse-diam-2-and-complement}.  Use that in Construction
  \ref{constr:d2-trip} with $r = xn$ for some $x$.  We will optimize the
  choice of $x$ at the end.  The total number of edges is then exactly
  $\binom{n}{2} + n + n (xn)$, because $G$ and its complement together
  contribute exactly $\binom{n}{2}$ edges, the $A$--$B$ matching
  contributes $n$ edges, and $n (xn)$ edges come from the complete
  bipartite graph between $B$ and $C$.

  On the other hand, each vertex in $B$ has degree at least $(1-2p)n + xn$,
  and each vertex in $C$ has degree equal to $n$.  Therefore, the sum of
  the squares of the vertex degrees is at least
  \begin{displaymath}
    (1-o(1))n(1+x)^2 n^2 + (xn)n^2
    =
    (1-o(1))n^3 (1 + 3x + x^2) .
  \end{displaymath}
  The ratio between this and the product of the numbers of vertices and
  edges is at least
  \begin{displaymath}
    (1-o(1))
    \frac{n^3 (1 + 3x + x^2)}{(2n + xn) \left( \frac{n^2+n}{2} + xn^2 \right)}
    =
    (1-o(1)) \frac{1 + 3x + x^2}{(2+x) \left(\frac{1}{2} + x\right)} .
  \end{displaymath}
  Substituting $x = 1$ gives a ratio of $\frac{10}{9} - o(1)$, proving
  Theorem \ref{thm:CHd2-false}, and it is easy to verify that this choice
  of $x$ maximizes the final function on the right hand side.
\end{proof}

\section{Caccetta-H\"aggkvist for higher diameter}
\label{sec:ch-diam-hi}

In this section, we prove Theorem \ref{thm:CH-diamk}, which resolves the
analogue of the Caccetta-H\"aggkvist conjecture in diameters $k \geq 3$.
The key concept which enables a number of our proofs is the following
definition.

\begin{defi}
  \label{defi:k-associated}
  In a graph $G$, an unordered vertex pair $\{x, y\}$ and an edge $e$ are
  said to be \emph{$k$-associated} if their distance $d_G(x, y)$ is at most
  $k$, but when $e$ is deleted, their distance $d_{G-e}(x, y)$ becomes
  greater than $k$.  A pair $\{x, y\}$ is called \emph{$k$-critical} if
  there exists some edge $e$ which is $k$-associated with $\{x, y\}$.
\end{defi}

Note that whenever an edge $e$ is $k$-associated with a pair $\{x, y\}$, it
immediately follows that $e$ is part of every shortest path between $x$ and
$y$.  When $k \geq 3$, there may be multiple shortest paths (all of the
same length).  It is convenient for us to have a single path to refer to
for each $\{x, y\}$, and so for each $k$-critical pair $\{x, y\}$, we
arbitrarily select one such shortest path and denote it $P_{xy}$.  Let
these $P_{xy}$ be called \emph{$k$-critical paths}.  (We will have exactly
one $k$-critical path per $k$-critical pair.)  For each edge $e$, let
$\mathcal{P}(e)$ be the set of all $k$-critical paths $P_{xy}$ such that
$\{x,y\}$ is $k$-associated with $e$.  Observe that by the above, $e$ is
always on every $P_{xy} \in \mathcal{P}(e)$, and in diameter-$k$-critical
graphs, $\mathcal{P}(e)$ is always nonempty.

In both this section and the next section, it will be useful to keep track
of the 3-vertex subgraph statistics.  For $0\le i\le 3$, let
$\mathcal{T}_i$ be the set of unordered triples $\{x,y,z\}$ in $V(G)$ such
that their induced subgraph $G[\{x,y,z\}]$ has exactly $i$ edges.  By
counting the number of pairs $(v,f)$ that vertex $v$ is not incident to
edge $f$, we see that
\begin{align}
  \label{ident:e(n-2)}
  m(n-2)=3|\mathcal{T}_3|+2|\mathcal{T}_2|+|\mathcal{T}_1|,
\end{align}
which, together with the fact that
\begin{align}
  \label{ident:deg2}
  \sum_{v} \binom{d_v}{2}=3|\mathcal{T}_3|+|\mathcal{T}_2|,
\end{align} implies that
\begin{align}
  \label{ident:deg2-en}
  \sum_{v} d_v^2-mn=3|\mathcal{T}_3|-|\mathcal{T}_1|.
\end{align}

We are now ready to prove the diameter-$k$-critical analogue of the
Caccetta-H\"aggkvist conjecture for $k \geq 3$.

\begin{proof}[Proof of Theorem \ref{thm:CH-diamk}, upper bound]
  Let $G$ be a diameter $k$-critical graph with $k \geq 3$.  Let
  $T=\{x,y,z\}\in \mathcal{T}_3$ be an arbitrary triangle in $G$. Writing
  $xy$ to denote the edge with endpoints $x$ and $y$, etc., select arbitrary
  $P_1\in \mathcal{P}(xy), P_2\in \mathcal{P}(yz)$ and $P_3\in
  \mathcal{P}(xz)$.  We claim that $P_1$, $P_2$, and $P_3$ all have length
  $k$.  Indeed, if, say, $P_1$ was a path of length at most $k-1$ from $u$
  to $v$ via $xy$, then by using $xz$ and $zy$ instead of $xy$, we obtain
  an alternate path of length at most $k$, contradicting the
  $k$-association of $xy$ and $\{u,v\}$.

  The path $P_1$ contains an edge adjacent to $xy$.  Without loss of
  generality, suppose that this edge is $tx$.  It is clear that $t \neq z$.
  Label the endpoints of $P_2$ by $u$ and $v$ such that $P_2$ traverses
  $u,y,z,v$ in that order.  We claim that $t\notin P_2$.  Indeed, if, say,
  $t \in uP_2y$, then since $t \neq y$, the path $u P_2 t x z P_2 v$ has
  length at most $k$ while avoiding the edge $yz$, contradicting the fact
  that $\{u, v\}$ is $k$-associated with $yz$.  Thus
  $|V(P_2)\cup \{t\} \setminus T|\ge k$.

  For each vertex $s\in V(P_2)\cup \{t\} \setminus T$, we choose an edge
  $f_s\in E(T)$ such that $s$ and $f_s$ form a triple $F_s\in
  \mathcal{T}_1$ as follows. For $t$, we have $ty\notin E(G)$ and $tz\notin
  E(G)$, as otherwise we could reroute $P_1$ between $y$ and $t$ either
  directly or via $yzt$, avoiding $yx$ entirely, contradicting $P_1 \in
  \mathcal{P}(xy)$.  So, the choice $f_t = yz$ produces $F_t = \{t, y, z\}
  \in \mathcal{T}_1$.  For any $s\in V(P_2)\setminus T$ which is
  between $u$ and $y$, clearly $sz\notin E(G)$. It also holds that
  $sx\notin E(G)$, as otherwise one could reroute $P_2$ via $x$ to avoid
  $yz$, while maintaining length at most $k$.  So $F_s:=\{s,x,z\}\in
  \mathcal{T}_1$ by choosing $f_s=xz$.  A similar argument handles the 
  $s\in V(P_2)\setminus T$ which are between $z$ and $v$.

  The above argument actually showed that for the triple $F_t$, $f_t = yz$
  and the edge $xy$ was on all shortest paths between $t$ and $y$.  For
  every other triple $F_s$, either $f_s = xz$ and $yz$ is on all shortest
  paths between $s$ and $z$, or $f_s = xy$ and $zy$ is on all shortest
  paths between $s$ and $y$.  In all of those cases, we have the property
  that
  \begin{quote}
    \emph{For any triple $F_s$, there exists an edge $f'\in E(T)-\{f_s\}$
  such that $f'$ is contained in all shortest $(s,s')$-paths, where
$s':=V(f')\cap V(f_s)$.}
  \end{quote}

  Let $\mathcal{F}(T)$ be the collection of triples which arise in this way
  from $T$, i.e., $\mathcal{F}(T) = \{F_s : s\in V(P_2)\cup \{t\} \setminus
T\}$.  We claim that $\mathcal{F}(T)\cap \mathcal{F}(T')=\emptyset$ for
distinct $T,T'\in \mathcal{T}_3$. To see this, assume for contradiction
that it is not the case. Then there exists $\{s,x,y\}\in \mathcal{F}(T)\cap
\mathcal{F}(T')$ for distinct triangles $T,T'$ such that $xy\in E(G)$.  As
$xy\in E(T)\cap E(T')$ by definition, let $T:=\{x,y,z\}$ and
$T':=\{x,y,z'\}$ for distinct vertices $z,z'$.  In light of the above
observation, we may assume that $xz$ is contained in all shortest
$(s,x)$-paths, and $yz'$ is contained in all shortest $(s,y)$-paths.  Let
$P$ be a shortest $(s,x)$-path and $P'$ be a shortest $(s,y)$-path, and
assume without loss of generality that $|P|\le |P'|$.  Then $sPzy$ is an
$(s,y)$-path of length $|P|\le |P'|$ which does not contain $yz'$, a
contradiction. This proves the claim.

Since $\mathcal{F}(T)$'s are disjoint subsets of $\mathcal{T}_1$ with at
least $k$ triples each, it follows that $|\mathcal{T}_1|\ge \sum_{T\in
\mathcal{T}_3} |\mathcal{F}(T)|\ge k|\mathcal{T}_3|\ge
3|\mathcal{T}_3|,$ completing the proof of the upper bound by \eqref{ident:deg2-en}.
\end{proof}

For the other part of our Caccetta-H\"aggkvist-type result for diameter 3
and higher, we must construct graphs which asymptotically approach our
bound.  We do this by selecting specific parameters for Construction
\ref{constr:dK}, and we include the following proof for completeness.

\begin{lem}
  \label{lem:constr:dK}
  Construction \ref{constr:dK} always produces a diameter-$k$-critical
  graph.
\end{lem}

\begin{proof}
  We first verify that every pair of vertices $\{x, y\}$ has distance at
  most $k$.  To see this, arbitrarily select vertices $u \in V_0$ and $v
  \in V_k$, allowing $\{u, v\}$ to possibly overlap with $\{x, y\}$.
  Observe that there is a path of length $k$ from $u$ to $v$ which passes
  through $x$ on the way, and a path of length $k$ from $v$ to $u$ which
  passes through $y$ on the way.  Therefore, there is a closed walk
  (possibly repeating vertices or edges) from $x$ back to itself via $y$,
  of length exactly $2k$, which implies that $x$ and $y$ are at distance at
  most $k$.

  To verify criticality, we have two types of edges.  Consider an edge $xy$
  in a matching, say with $x \in V_i$ and $y \in V_{i+1}$ where $1 \leq i
  \leq k-2$.  If $xy$ is deleted, then $y$ can only reach $V_0$ by going
  all the way to $V_k$ and then coming back, and so the distance between
  $y$ and any vertex of $V_0$ rises above $k$.  Now consider the other kind
  of edge $xy$, where $x \in V_0$ and $y \in V_1$, say. Let $P$ be the unique 
  path of length $k-2$ from $V_1$ to $V_{k-1}$ with endpoint $y$, and let $y'\in V_{k-1}$
  be the other endpoint of $P$. If edge $xy$ is deleted, then one may verify that 
  the distance between $x$ and $y'$ changes from $k-1$ to $k+1$, which completes our proof.  
\end{proof}

We now use Construction \ref{constr:dK} to prove our lower bound on
Caccetta-H\"aggkvist for diameter-$k$-critical graphs.

\begin{proof}[Proof of Theorem \ref{thm:CH-diamk}, lower bound]
  Let $k\ge 3$ be fixed.  We build $n$-vertex graphs by using Construction
  \ref{constr:dK} with $a = 1$, $b = 1$, and $c = n-k$.  (We actually can
  select any sub-linear function $b = o(n)$, e.g., $b = \sqrt{n}$, to
  create a wider variety of asymptotically extremal constructions.)  Then,
  \begin{displaymath}
    \sum_v d_v^2 = (1+o(1)) bn^2 ,
  \end{displaymath}
  and the number of edges is $(1+o(1)) bn$, and so the ratio between
  $\sum d_v^2$ and the product of the numbers of vertices and edges indeed
  tends to 1 as $n \rightarrow \infty$.
\end{proof}

\section{Caccetta-H\"aggkvist upper bound for diameter 2}

In this section, we prove Theorem \ref{thm:CHd2-improve}, which improves
the upper bound on the constant for the Caccetta-H\"aggkvist problem in the
diameter-2 setting.  Let $c$ and $N$ represent sufficiently small and sufficiently large absolute
constants, respectively, throughout this section. We will make a series of claims which
hold for large $N$ and positive constants $c_i$, where $c_0:=c$ and $c_i$ is a function
of variables $c_0,c_1,...,c_{i-1}$ tending to $0$ as $c\to 0$. The eventual
values of $N$ and $c_i$ can be explicitly calculated, although some of them will
not be expressed in the proof to keep the main ideas clean.

Let $G$ be an arbitrary diameter-2-critical graph with $n$ vertices and $m$
edges, where $n\ge N$. We will show that
\begin{align}
  \label{ineq:6/5-c}
  \sum_{v\in V(G)} d_v^2\le \left(\frac{6}{5}-c\right) nm .
\end{align}
We use the notion of 2-critical paths from the beginning of Section \ref{sec:ch-diam-hi}, 
and following that section, we also define
$\mathcal{P}(e)$ as the set of all 2-critical paths $P_{xy}$ such that
$\{x, y\}$ is 2-associated with the edge $e$.  Since our diameter is always 2 in
this section, we will write \emph{associated}\/ and \emph{critical path}\/
to refer to the concepts of 2-associated and 2-critical path.

\begin{defi}
  \label{defi:foot}
  Let $T \in \mathcal{T}_3$ be a triangle.  We say that a vertex $v \not
  \in T$ is a \emph{foot}\/ of $T$ if there are some $x, y \in T$ such that
  the path $vxy$ belongs to $\mathcal{P}(xy)$.
\end{defi}

\begin{lem}
  \label{lem:feet>=2}
  Every triangle $T \in \mathcal{T}_3$ has at least 2 feet, and if $v$ is a
  foot of $T$, then it is adjacent to exactly one vertex of $T$.
\end{lem}

\begin{proof}
  Consider any triangle $T:=\{x,y,z\}\in \mathcal{T}_3$. For any edge (say
  $xy$) in $T$, every path $P \in \mathcal{P}(xy)$ must have length 2
  (suppose it is $P = vxy$).  Since the removal of $xy$ is supposed to
  increase the distance between $v$ and $y$ above 2, we must have
  $vy,vz\notin E(G)$, as claimed.  Now consider $\mathcal{P}(yz)$.  This
  must contain a path of length 2, and the outside vertex cannot be $v$
  because $v$ is adjacent to neither of $\{y, z\}$, so it produces another
  foot of $T$.
\end{proof}

\begin{defi}
  Given a triangle $T \in \mathcal{T}_3$, let $\mathcal{F}(T)$ be the
  collection of all triples $\{v,y,z\}$ in $\mathcal{T}_1$ where $v$ is a
  foot of $T$, both $y$ and $z$ are in $T$, and $v$ is not adjacent to
  either of $\{y, z\}$.
\end{defi}

\begin{lem}
  \label{lem:feet-unique}
  For distinct triangles $T,T'\in \mathcal{T}_3$, $\mathcal{F}(T)$ and
  $\mathcal{F}(T')$ are disjoint.
\end{lem}

\begin{proof}
  Assume for contradiction that $\{v,y,z\}\in \mathcal{F}(T)\cap
  \mathcal{F}(T')$ such that $yz\in E(G)$.  Then $T=\{x,y,z\}$ and
  $T'=\{x',y,z\}$ for distinct vertices $x,x'$. Without loss of generality,
  assume that $vxy$ is a critical path in $\mathcal{P}(xy)$.  However,
  $vx'y$ is also an $(v,y)$-path of length 2, which does not contain
  the edge $xy$, a contradiction.
\end{proof}

\begin{lem}
  Let $\mathcal{T}_3^*$ be the set of triangles with at least three feet.
  To prove \eqref{ineq:6/5-c} and hence Theorem \ref{thm:CHd2-improve}, 
  it suffices to establish any one of the following two conditions:
  \begin{displaymath}
    {\bf (C1).}
    \quad
    |\mathcal{T}_2| \ge \frac{5c}{2}\cdot nm
    \quad\quad\quad\quad
    {\bf (C2).}
    \quad
    |\mathcal{T}_3^*| \ge \frac{5c}{6}\cdot nm.
  \end{displaymath}
\end{lem}

\begin{proof}
  By Lemma \ref{lem:feet-unique} and Lemma \ref{lem:feet>=2}, we have
  \begin{align}
    \label{ineq:T1T3}
    |\mathcal{T}_1|
    \ge
    \sum_{T\in \mathcal{T}_3} |\mathcal{F}(T)|
    \ge
    2|\mathcal{T}_3|+|\mathcal{T}_3^*|.
  \end{align}
  By \eqref{ident:e(n-2)} and \eqref{ineq:T1T3}, we see that
  \begin{align}
    \nonumber
    m(n-2) &= 3|T_3| + 2|T_2| + |T_1| \\
    \nonumber
    mn &\ge 3|T_3| + 2|T_2| + (2|T_3| + |T_3^*|) + 2m \\
    \label{ineq:mn>=5t3}
    mn &\ge 5|\mathcal{T}_3|+|\mathcal{T}_3^*|+2|\mathcal{T}_2| + 2m .
  \end{align}
  Also, by doubling \eqref{ident:deg2}, we find
  \begin{align}
    \nonumber
    \sum_v d_v^2 - \sum_v d_v
    &= 6|T_3| + 2|T_2| \\
    \nonumber
    \sum_v d_v^2 &= 6|T_3| + 2|T_2| + 2m \\
    \label{ineq:dv2<=t3}
    \sum_v d_v^2 &\le mn + |T_3| - |T_3^*| ,
  \end{align}
  where we used \eqref{ineq:mn>=5t3} for the last deduction.

  Now suppose (C1) holds. Then by \eqref{ineq:mn>=5t3}, we have $mn\ge
  5|\mathcal{T}_3|+5c nm$ and thus \eqref{ineq:dv2<=t3} implies that
  $\sum_{v}d_v^2-nm\le |\mathcal{T}_3|\le (1/5-c) nm$, giving
  \eqref{ineq:6/5-c}.  On the other hand, if (C2) holds, by
  \eqref{ineq:dv2<=t3} we have $\sum_{v}d_v^2-nm\le
  |\mathcal{T}_3|-\frac{5c}{6}\cdot nm$.  To compare $|\mathcal{T}_3|$ with
  $nm$, we use \eqref{ineq:mn>=5t3} to obtain $nm\ge
  5|\mathcal{T}_3|+\frac{5c}{6}\cdot nm$, and thus $\sum_{v}d_v^2-nm \le
  (1/5-c) nm$, which again implies \eqref{ineq:6/5-c}.
\end{proof}

Our objective is now to show that at least one of (C1) and (C2) always
holds, unless \eqref{ineq:6/5-c} holds directly.  To this end, we first show 
that almost all edges have endpoints with similar neighborhoods.  
For the remainder of this section, we write $A\Delta B$ to denote 
the symmetric difference of sets $A$ and $B$.

\begin{lem}
  Define $c_1:=c^{1/4}$, and let $\mathcal{E}_1$ be the set of all edges
  $uv$ such that $|N_u \Delta N_v| \le c_1n$. If (C1) does not hold, then
  $|\mathcal{E}_1| \ge (1-c_1)m$.
  \label{lem:CHd2upClaim1}
\end{lem}

\begin{proof}
  Suppose on the contrary that there are at least $c_1 m$ edges $uv$
  satisfying $|N_u\Delta N_v|> c_1n$, but (C1) does not hold. Note that for
  any vertex $w\in N_u\Delta N_v$, the set $\{w,u,v\}\in \mathcal{T}_2$.
  Then $|\mathcal{T}_2|\ge \frac{(c_1m)(c_1n)}{2}> \frac{5c}{2}\cdot
  mn$, and so (C1) holds, contradiction.
\end{proof}

\begin{lem}
  Let $\mathcal{E}_2$ be the set of edges $uv \in \mathcal{E}_1$ such that
  $u$ and $v$ have at least $(1/2+c_1)n$ common neighbors.  If (C1) does not hold and neither does \eqref{ineq:6/5-c},
  then $|\mathcal{E}_2|\ge c_1m$.
  \label{lem:CHd2upClaim2}
\end{lem}

\begin{proof}
  Suppose for contradiction that $|\mathcal{E}_2|< c_1m$. By Lemma \ref{lem:CHd2upClaim1}, 
  it holds that $|\mathcal{E}_1|\ge (1-c_1)m$, and so, for sufficiently small constant $c$, we have
  \begin{align*}
    \sum_{v}d_v^2
    &=
    \sum_{uv\in E(G)}(d_u+d_v)=\sum_{uv\in E(G)}(|N_u\Delta N_v|+2|N_u\cap N_v|)\\
    &\le \sum_{uv\in \mathcal{E}_1}(c_1n+2|N_u\cap N_v|)+\sum_{uv\notin \mathcal{E}_1}2n\\
    &\le 3c_1 mn+2\sum_{uv\in \mathcal{E}_2}|N_u\cap N_v|+2\sum_{uv\in
      \mathcal{E}_1 \setminus \mathcal{E}_2}|N_u\cap N_v|\\
      &\le 3c_1 mn+2c_1mn+2m(1/2+c_1)n=(1+7c_1)
      mn<\left(\frac{6}{5}-c\right) mn,
  \end{align*}
  that is, \eqref{ineq:6/5-c} holds, contradicting the assumption. 
\end{proof}

\begin{lem}
  Define $c_2:=\sqrt{c}/4$.  If (C1) does not hold, then $m \ge
  \frac{c_1}{2}n^2$ and $|\mathcal{E}_2|\ge c_2 n^2$.
  \label{lem:CHd2upClaim3}
\end{lem}

\begin{proof}
  Let $H$ be the subgraph of $G$ spanned by the edges of $\mathcal{E}_2$,
  and let $h=|V(H)|/n$. Note that for $v\in V(H)$, $d_G(v)\ge (1/2+c_1)n$.
  If $h\ge \frac{1}{4}$, then $m\ge \frac{1}{2}\sum d_G(v)\ge
  \frac{n^2}{16}$ and thus by Lemma \ref{lem:CHd2upClaim2},
  $|\mathcal{E}_2|\ge c_1m\ge \frac{c_1}{16}n^2>c_2n^2$ as desired.

  It remains to consider $h< \frac{1}{4}$. Every $v \in V(H)$ has at least
  $(1/2+c_1-h) n$ neighbors out of $V(H)$, implying that $m\ge
  |V(H)|\cdot(1/2+c_1-h)n=h(1/2+c_1-h)n^2$.  Using Lemma
  \ref{lem:CHd2upClaim2}, we find $\frac{(hn)^2}{2}\ge |\mathcal{E}_2|\ge
  c_1m\ge c_1h(1/2+c_1-h)n^2$, which shows that $(h-c_1)(1+2c_1)\ge 0$.
  Thus $\frac{1}{4}> h\ge c_1$, implying that $m\ge \frac{c_1}{4}n^2$ and
  $|\mathcal{E}_2|\ge c_1^2n^2/4=c_2n^2$. 
\end{proof}

\begin{lem}
  Define $c_3:=40 \sqrt{c}$. If (C1) does not hold, then there exists an edge
  $uv\in \mathcal{E}_2$ such that $N_u\cap N_v$ induces at least
  $(1-c_3)\binom{|N_u\cap N_v|}{2}$ edges.
  \label{lem:CHd2upClaim4}
\end{lem}

\begin{proof}
  By Lemma \ref{lem:CHd2upClaim3}, there exists a matching $M$ with edges
  from $\mathcal{E}_2$ of size at least $\frac{c_2}{2}n$.  Suppose for
  contradiction that for every edge $uv\in M$, $N_u\cap N_v$ has at least
  $c_3\binom{|N_u\cap N_v|}{2}$ non-adjacent pairs.  Note that any
  non-adjacent pair $\{x,y\}$ in $N_u\cap N_v$ contributes two triples
  $\{u,x,y\}$ and $\{v,x,y\}$ to $\mathcal{T}_2$.  Therefore, we see that
  \begin{displaymath}
    |\mathcal{T}_2|\ge 2 c_3\binom{(1/2+c_1)n}{2} |M|\ge
  \frac{c_2c_3}{8}n^3\ge \frac{c_2c_3}{4}mn\ge \frac{5c}{2}mn,
  \end{displaymath}
  that is, (C1) holds.
\end{proof}

\begin{lem}
  Define $c_4:=\sqrt{40}c^{1/4}$, suppose (C1) does not hold, and let $uv$
  be the edge from Lemma \ref{lem:CHd2upClaim4}. Then $G[N_u\cap N_v]$ has
  an induced subgraph $D$ with at least $(1/2-c_4) n$ vertices and minimum
  degree at least $(1-c_4) |D|$.
  \label{lem:CHd2upClaim5}
\end{lem}

\begin{proof}
  Start with $D = G[N_u \cap N_v]$.  As long as $D$ has a vertex of degree
  less than $(1-\sqrt{c_3})|D|$, delete it.  At every stage of this
  procedure, $\sqrt{c_3}|D|\cdot(|N_u\cap N_v|-|D|)$ is at most the number
  of non-adjacent pairs in $N_u\cap N_v$.  Our assumption on $uv$ from
  Lemma \ref{lem:CHd2upClaim4} then implies that
  \begin{displaymath}
    \sqrt{c_3}|D|\cdot(|N_u\cap N_v|-|D|)
    \le
    \frac{c_3}{2}|N_u\cap N_v|^2
  \end{displaymath}
  must hold throughout the process.  It is clear that this holds when $D =
  N_u \cap N_v$, and since the function $f(x) = x(M-x)$ is a
  downward-opening parabola, the process must stop well before $D$ reaches
  $\frac{1}{2} |N_u \cap N_v|$.  So, throughout the process,
  \begin{align*}
    \sqrt{c_3}\frac{|N_u\cap N_v|}{2} \cdot(|N_u\cap N_v|-|D|)
    &\le
    \frac{c_3}{2}|N_u\cap N_v|^2 \\
    |N_u\cap N_v| - |D|
    &\le
    \sqrt{c_3} |N_u\cap N_v|,
  \end{align*}
  and we conclude that at the end, $|D|\ge
  \left(1-\sqrt{c_3}\right)|N_u\cap N_v|\ge (1/2-c_4)\cdot n$.
\end{proof}

\begin{defi}
  Let $D$ be the subgraph produced by Lemma \ref{lem:CHd2upClaim5}.  For
  each edge $e \in E(D)$, since the endpoints of $e$ form a triangle with
  $u$, there exists a path of length 2 in $\mathcal{P}(e)$, say $x'xy$
  where $xy = e$, $x' \notin D$, and $x'y \notin E(G)$.  We call any such
  $x'$ an \emph{arm} of edge $e$.  Define the digraph $\overrightarrow{D}$
  on $V(D)$ as follows: for any edge $xy \in E(D)$, we place a directed
  edge from $y$ to $x$ to $\overrightarrow{D}$ whenever $xy$ has an arm
  $x'$ such that $xx'\in E(G)$.
\end{defi}

Note that each edge in $D$ can produce either one or both directed edges in
$\overrightarrow{D}$, i.e., between each pair of vertices, there can be
zero, one, or two directed edges.

\begin{lem}
  Define $c_5:=5\sqrt{c_4}$, and suppose (C1) does not hold.  Let $S$ be
  the set of vertices in $\overrightarrow{D}$ with in-degree at least $4c_4
  |D|$. Then, for large $n\ge N$, $|S| > (1/2-c_5) n$.
  \label{lem:CHd2upClaim6}
\end{lem}

\begin{proof}
  Let $S^c$ be the complement of $S$ in $V(D)$ and let $s=|S^c|/n$.  Let
  $d^-(v)$ denote the in-degree of $v$ in $\overrightarrow{D}$. By Lemma
  \ref{lem:CHd2upClaim5}, the number of non-adjacent pairs in
  $\overrightarrow{D}$ is less than $\frac{c_4}{2}|D|^2$, thus
  \begin{displaymath}
    \binom{sn}{2}- \frac{c_4}{2}|D|^2
    \le
    e(\overrightarrow{D}[S^c])
    \le
    \sum_{v\in S^c} d^-(v)\le 4c_4|D|\cdot sn.
  \end{displaymath}
  Thus, for large $n\ge N$, we have $\frac{s^2}{2} \le \frac{c_4}{2}+4c_4s +\frac{s}{2n}< 5c_4$, implying
  that $s < 4\sqrt{c_4}$. By Lemma \ref{lem:CHd2upClaim5}
  again, we see $|S|=|D|-|S^c|\ge (1/2-c_4-4\sqrt{c_4}) n > (1/2-c_5) n$.
\end{proof}

\begin{defi}
 For each $x\in S$ and $x'\notin V(D)$ satisfying $xx'\in E(G)$, define
 $A_{x'}(x)$ to the set of all vertices $y\in V(D)$ such that $x'xy \in
 \mathcal{P}(xy)$.  We say $x\in S$ is {\em rich} if there exists some
 $x'\notin V(D)$ such that $|A_{x'}(x)|\ge 2c_4|D|$.
\end{defi}

Observe that since any edge $\overrightarrow{yx} \in \overrightarrow{D}$
has an arm $x'\notin V(D)$ such that $x'xy\in \mathcal{P}(xy)$, the union
$\bigcup_{x'} A_{x'}(x)$ is just the in-neighborhood $N^-(x)$ of $x$ in
$\overrightarrow{D}$, which is of at least $4c_4 |D|$ by Lemma
\ref{lem:CHd2upClaim6}.

\begin{lem}
  Suppose (C1) does not hold, and let $X$ be the set of rich vertices in
  $S$. Then $|X|\ge (1-\sqrt{c}) |S|$ or (C2) holds.
  \label{lem:CHd2upClaim7}
\end{lem}

\begin{proof}
  Suppose for contradiction that there are at least $\sqrt{c} |S|$ vertices
  $x\in S$ such that $|A_{x'}(x)|<2c_4|D|$ for every $x'\notin V(D)$
  adjacent to $x$.  Fix such a vertex $x$. Consider any $y\in N^-(x)$, say
  $y\in A_{x_1}(x)$ for $x_1\notin V(D)$. By Lemma \ref{lem:CHd2upClaim5},
  $y$ is non-adjacent at most $c_4|D|$ vertices in $D$. Since $|N^-(x)|\ge
  4c_4|D|$ (by Lemma \ref{lem:CHd2upClaim6}), we see that there are more
  than $c_4|D|$ vertices $y'\in N^-(x) \setminus A_{x'}(x)$ such that
  $yy'\in E(D)$.

  We need the following property: for distinct $x_1,
  x_2\notin V(D)$,
  \begin{equation}
    \label{prop:formT3*}
    \text{If } y_1\in A_{x_1}(x) \text{ and } y_2\in A_{x_2}(x) \text{ such
    that } y_1y_2\in E(D), \text{ then }  \{y_1,y_2,x\}\in \mathcal{T}_3^*.
  \end{equation}
  To see this, it is clear that $\{y_1,y_2,x\}$ is in $\mathcal{T}_3$ and
  $x_1,x_2$ are two feet of it; we can find a third foot of $\{y_1,y_2,x\}$
  by considering the edge $y_1y_2$.

  By \eqref{prop:formT3*}, every such $\{y,y',x\}\in \mathcal{T}_3^*$, and
  there are at least $\frac{1}{2}|N^-(x)|\cdot c_4|D|$ pairs $\{y,y'\}$.
  Thus, for every such vertex $x$, there are at least
  $\frac{1}{2}|N^-(x)|\cdot c_4|D| > (c_4|D|)^2$ triples in
  $\mathcal{T}_3^*$ which contain $x$.  So, by Lemmas
  \ref{lem:CHd2upClaim5} and \ref{lem:CHd2upClaim6}, we have
  \begin{displaymath}
    |\mathcal{T}_3^*|
    \ge
    \frac{1}{3} \cdot \sqrt{c}|S|\cdot (c_4|D|)^2
    >
    \frac{1}{3} \sqrt{c}(1/2-c_5)c_4^2(1/2-c_4)^2n^3
    \ge
    \frac{\sqrt{c}}{13}\cdot c_4^2 mn
    >
    \frac{5c}{6}\cdot mn,
  \end{displaymath}
  that is, (C2) holds.
\end{proof}

\begin{lem}
  Suppose (C1) and (C2) do not hold.  For each vertex $x\in X$, choose and
  fix an adjacent vertex $x'\notin V(D)$ for which $|A_{x'}(x)|\ge
  2c_4|D|$.  Then, every such $x'$ is not adjacent to any vertex in its
  corresponding $X \setminus \{x\}$.
  \label{lem:CHd2upClaim8}
\end{lem}

\begin{proof}
  Suppose for contradiction that $x'$ is adjacent to some $y\in X \setminus
  \{x\}$.  By Lemma \ref{lem:CHd2upClaim5}, $y$ is non-adjacent to at most
  $c_4|D|$ vertices in $D$. Thus there exists some vertex $z\in A_{x'}(x)$
  such that $yz\in E(D)$. Then $zyx'$ is a $(z,x')$-path not containing
  edge $zx$.  On the other hand, $z\in A_{x'}(x)$ implies that $zxx'\in
  \mathcal{P}(zx)$, a contradiction.
\end{proof}

We are now ready to combine all of the above steps to complete the proof of
our improved upper bound for the diameter-2-critical Caccetta-H\"aggkvist
conjecture.

\begin{proof}[Proof of Theorem \ref{thm:CHd2-improve}]
  Let $Y$ be the set of all $x'$ as defined in Lemma
  \ref{lem:CHd2upClaim8}.  Note that by construction, $Y$ is disjoint from
  $X$, and by that lemma, the bipartite subgraph $G[X,Y]$ induces a perfect
  matching, so
  \begin{displaymath}
    |Y|
    =
    |X|
    >
    (1-\sqrt{c})(1/2-c_5)n
    >
    (1/2-20c^{1/8}) n.
  \end{displaymath}
  Let $Z=V(G) \setminus (X\cup Y)$, and define $t$ such that
  $|X|=|Y|=(1/2-t) n$ and $|Z|=2t n$ for some $0\le t\le 20c^{1/8}$. Then
  for every $x\in X$, by Lemma \ref{lem:CHd2upClaim5}, we have that
  \begin{displaymath}
    (1/2-10c^{1/4})n
    < (1-c_4)(1/2-c_4)n
    \le d_G(x)
    \le (1/2+t)n.
  \end{displaymath}
  %Since $G[X, Y]$ induces a perfect matching and $|Z| \leq 40 c^{1/8} n$,
  %for each $x \in X$, its number of neighbors in $X$ is $d_X(x) \geq (1/2 -50 c^{1/8}) n$.  
  So the total number of edges satisfies
  \begin{align*}
    m
    &\ge
    \frac{1}{2}\sum_{x\in X}d_G(x)+e(Y) \\
    &\ge
    \frac{1}{2}\left(\frac{1}{2}-10c^{1/4}\right)
    \left(\frac{1}{2}-t\right)n^2+e(Y)
    \\
    &\ge
    \left(\frac{1}{8}-10c^{1/8}\right)n^2+e(Y).
  \end{align*}
  Also, since every vertex in $X$ or $Y$ has degree at most $(1/2+t)n$,
  \begin{align*}
    \sum_{v}d_v^2
    &\le
    |X| \left(\frac12+t\right)^2n^2+|Z|n^2+\left(\frac12+t\right)n\sum_{y\in Y} d_y\\
    &=
    \left(\left(\frac12-t\right)\left(\frac12+t\right)^2+2t\right) n^3
    +\left(\frac{1}{2}+t\right)n\cdot \left(2e(Y)+e(X,Y)+e(Y,Z)\right)\\
    &\le
    \left(\left(\frac12-t\right)\left(\frac12+t\right)^2+2t\right) n^3
    +(1+2t)n\cdot e(Y)
    +
    \left(\frac{1}{2}+t\right)n\cdot\left(\frac{1}{2}-t\right)n \cdot
    (1+2tn)\\
    &\le
    \left(\frac{1}{8}+4t+\frac{1}{4n}\right) n^3 +(1+2t)n\cdot e(Y),
  \end{align*}
  where $t\le 20c^{1/8}$.  Therefore, it is clear that by choosing $c$
  sufficiently small and $N$ sufficiently large, given $n\ge N$, we will have
  \begin{displaymath}
    \sum_{v}d_v^2\le \left(\frac{6}{5}-c\right) mn,
  \end{displaymath}
   completing the proof.
\end{proof}

\section{Asymptotics for maximizing edges}

Construction \ref{constr:dK} established a family of diameter-$k$-critical
graphs with $\frac{n^2}{4(k-1)}+o(n^2)$ edges.  In this section, we show
that estimate is tight up to a constant factor by proving Theorem
\ref{thm:diamk-edges}, which upper-bounds the number of edges by
$\frac{3n^2}{k}$.

Let $k \ge 2$ be fixed throughout this section.  Recall from Definition
\ref{defi:k-associated} that an unordered pair $\{x, y\}$ and edge $e$ are
$k$-associated if the distance between $x$ and $y$ is at most $k$, but rises above $k$ if $e$
is deleted.  In such a situation, if $P$ is an $(x, y)$-path of length at
most $k$, we also say that $P$ and $e$ are $k$-associated.

\begin{lem}
  \label{lem:k/3-assoc}
  Let $G$ be a diameter-$k$-critical graph.  For any edge $e$, there exists
  a path $P$ of length $\lceil\frac{k}{3}\rceil$ such that $e$ is
  $\lceil\frac{k}{3}\rceil$-associated with $P$.
\end{lem}

\begin{proof}
  For any edge $e$, since $G$ is diameter-$k$-critical, there exists a pair
  $\{u,v\}$ such that $d_G(u,v)\le k$ and $d_{G-e}(u,v)>k$.  Let $L$ be a
  shortest $(u,v)$-path.  If $L$ has length at least
  $\lceil\frac{k}{3}\rceil$, then we can choose vertices $x,y\in V(L)$ such
  that $e\in xLy$ and $xLy$ has length $\lceil\frac{k}{3}\rceil$.  We claim
  that $e$ and $xLy$ are $\lceil\frac{k}{3}\rceil$-associated.  Indeed, if
  not, then there must exist an $(x,y)$-path $M$ in $G-e$ of length at most
  $\lceil\frac{k}{3}\rceil$, and one can use $M$ and $L$ together to
  construct a path of length at most $k$ from $u$ to $v$ avoiding $e$,
  contradiction.

  It thus suffices to consider the case when $L$ has length at most
  $\lceil\frac{k}{3}\rceil-1$.  Write $e=ab$ and consider the
  depth-first-search tree $T$ with root $a$.  We see that the depth of $T$
  is at least $\lceil\frac{k}{2}\rceil$, as otherwise $d_G(s,t)\le
  d_G(a,s)+d_G(a,t)\le 2(\lceil\frac{k}{2}\rceil -1)<k$ for all pairs
  $\{s,t\}$, contradicting the fact that $G$ has diameter $k$.  Thus there
  exists a path $P'$ from the root $a$ to some vertex, say $z$, such that
  $P'$ has length $\lceil\frac{k}{3}\rceil$.  We define $P:= P'$ if $e\in
  P'$ and $P:=P'\cup \{e\}-z$ otherwise.  Note that $P$ is a path of length 
  $\lceil\frac{k}{3}\rceil$, satisfying $e\in P$.

  We will show that $e$ is $\lceil\frac{k}{3}\rceil$-associated with $P$.
  Let $x,y$ be the endpoints of $P$.  Suppose on the contrary that
  $\{x,y\}$ and $e$ are not $\lceil\frac{k}{3}\rceil$-associated; then
  there exists an $(x,y)$-path $Q$ in $G-e$ such that $|Q|\le
  \lceil\frac{k}{3}\rceil$.  Then, by combining $P$ and $Q$, we find that
  there is a walk of length at most $2\lceil\frac{k}{3}\rceil - 1$ from one
  endpoint of $e$, to $x$, to $y$, and then to the other endpoint of $e$,
  completely avoiding $e$.  Therefore, if we follow $L$ from $u$ to the
  nearest endpoint of $e$, and then take this $e$-avoiding walk to the
  other endpoint of $e$, and finish along $L$ to $v$, we find an
  $e$-avoiding walk from $u$ to $v$ of length at most $3
  \lceil\frac{k}{3}\rceil-2\le k$.  This contradicts $d_{G-e}(u,v)>k$, and
  completes the proof.
\end{proof}

We are now ready to show that every diameter-$k$-critical graph has at most
$\frac{3n^2}{k}$ edges.

\begin{proof}[Proof of Theorem \ref{thm:diamk-edges}.]

  Consider an arbitrary diameter-$k$-critical graph $G$. Let us 
  say that a pair $\{x,y\}$ and an edge $e$ are {\em matched} in $G$
  if $e$ is incident to at least one of $x,y$ and all shortest
  $(x,y)$-paths contain $e$.  We will estimate the number $N$ of pairs
  $(\{x,y\},e)$ such that $\{x,y\}$ and $e\in E(G)$ are matched in $G$.

  By the definition, it is clear that any pair $\{x,y\}$ can only match at
  most two edges, implying that $N\le 2\binom{n}{2}\le n^2$.  On the other
  hand, by Lemma \ref{lem:k/3-assoc}, any edge $ab$ is
  $\lceil\frac{k}{3}\rceil$-associated with a path $P$ of length
  $\lceil\frac{k}{3}\rceil$.  Without loss of generality, suppose that $P$
  is an $(x,y)$-path such that $x,a,b,y$ appear on $P$ in order, where
  possibly $x=a$ or $b=y$.  Then for any vertex $z\in xPa$, the edge $ab$
  is matched with $\{z,b\}$, and for any vertex $z\in bPy$, the edge $ab$
  is matched with $\{a,z\}$.  This shows that every edge is matched with at
  least $|V(P)|-1=\lceil\frac{k}{3}\rceil$ many pairs, implying
  $\frac{k}{3} e(G)\le N\le n^2$ and therefore $e(G)\le \frac{3n^2}{k}$.
\end{proof}

\section{Diameter 3}
\label{sec:diam3}

Throughout this section, let $G$ be a diameter-3-critical graph on $n$
vertices.  We will prove that $G$ has at most $\frac{n^2}{6} + o(n^2)$ edges.
Since we will work exclusively with diameter-3 graphs, let us simply say
that a pair $\{x,y\}$ of vertices and an edge $e$ are {\em associated} if
$d_G(x,y)\le 3$ and $d_{G-e}(x,y)\ge 4$.  Similarly, say that a pair
$\{x,y\}$ is {\em critical} if there exists some edge $e$ associated with
$\{x,y\}$.  For any critical pair $\{x,y\}$, we arbitrarily select one of
the $(x,y)$-paths of the smallest length to be $P_{xy}$, its corresponding
{\em critical path}.  We refer to a critical pair $\{x,y\}$ and its
corresponding critical path $P_{xy}$ interchangeably, i.e., we also say $e$
is associated with $P_{xy}$ if $e$ is associated with $\{x,y\}$.  Note that
$P_{xy}$ must be of length at most 3 and contain all edges associated with
$\{x,y\}$.

For every edge $e$, let $\mathcal{P}_1(e)$ be the set containing all
critical paths $P_{xy}$ associated with $e$.  Since $G$ is diameter
3-critical, $\mathcal{P}_1(e)\neq \emptyset$ for every edge $e$. A pair
$\{x,y\}$ is {\em 2-critical} if there exists a unique $(x,y)$-path of
length at most 2, and such a path is also called {\em 2-critical}.  For
every edge $e$, let $\mathcal{P}_2(e)$ be the set of all 2-critical paths
containing $e$.  The {\em multiplicity} of an edge $e$ is defined as
$m(e):=|\mathcal{P}_1(e)|+|\mathcal{P}_2(e)|$.

\subsection{Critical paths with all edges associated}
\label{sec:diam3:Pt}

Let $\mathcal{P}$ be the set of all critical paths $P_{xy}$ that are
associated with at least 2 edges.  Since every $(x, y)$-path of length at
most 3 must then contain those 2 edges, and $P_{xy}$ has length at most 3,
every edge in $P_{xy}$ must actually be associated with $P_{xy}$, and thus
$P_{xy}$ is the unique $(x,y)$-path of length at most 3.  For a positive
integer $t$, let $\mathcal{P}_t$ be the set of all critical paths in
$\mathcal{P}$ with length 3, where the middle edge has multiplicity at
least $t$ and the two non-middle edges each have
multiplicity less than $t$.

Inspired by the proof of F\"uredi \cite{Fu} for the diameter-2 case, we use
the ``(6, 3)'' theorem of Ruzsa and Szemer\'edi \cite{RSz} to show that
$|\mathcal{P}_t|$ is a lower order term compared to $n^2$.  Recall that a
{\em 3-uniform hypergraph} $H$ is a pair $(V(H),E(H))$, where the {\em
edge-set} $E(H)$ is a collection of 3-element subsets of $V(H)$, each of
which is called a \emph{3-edge}. $H$ is {\em linear} if any two distinct
3-edges share at most one vertex.  In a linear 3-uniform hypergraph, three
3-edges form a {\em triangle} if they form a structure isomorphic to
$\{\{1,2,3\},\{3,4,5\},\{5,6,1\}\}$.  Let $\RS(n)$ be the maximum number of
3-edges in a triangle-free, linear 3-uniform hypergraph on $n$ vertices.
\begin{thm} (Ruzsa and Szemer\'edi \cite{RSz})
\label{thm:RSz}
$\RS(n)=o(n^2)$.
\end{thm}

The proof of the following lemma parallels F\"uredi's, differing mainly at
the definition of the auxiliary hypergraph.  We write out the proof for
completeness.

\begin{lem}
\label{lem:P_t}
For each positive integer $t$, $|\mathcal{P}_t|\le 54t\cdot\RS(n)$.
\end{lem}
\begin{proof}
  Consider a path $P_{xy}:=xaby$ in $\mathcal{P}_t$. By definition,
  multiplicities $m(xa)<t$, $m(ab)\ge t$, $m(by)<t$, and edges $xa,ab,by$
  are all associated with $P_{xy}$, implying that $xab$ and $aby$ are
  2-critical paths.  Define the 3-uniform hypergraph $H_1$ such that
  $V(H_1):=V(G)$, and form the edge-set $E(H_1)$ by arbitrarily choosing
  exactly one of $\{x,a,y\}$ and $\{x,b,y\}$ for each path $xaby$ in
  $\mathcal{P}_t$, so that $|E(H_1)|=|\mathcal{P}_t|$.  For a 3-edge
  $\{x,a,y\}\in E(H_1)$ obtained from the path $P_{xy}=xaby\in
  \mathcal{P}_t$, we call vertices $a$ and $x$ the {\em center} and {\em
  handle} of this 3-edge, respectively.

  We claim that the number of 3-edges of  $H_1$  intersecting  $\{x,a,y\}$
  in 2 elements is at most $2t-2$.  To see this, first observe that the
  critical pair $\{x,y\}$ does not appear in any other 3-edges of $H_1$.
  Since $m(xa)<t$, edge $xa$ (and hence the pair $\{x,a\}$) is contained in
  fewer than $t$ critical paths in $\mathcal{P}$, implying that the number
  of 3-edges of $H_1$ containing $\{x,a\}$ is fewer than $t$.  Also note
  that $aby$ is 2-critical, so if $\{a,y\}$ is contained in some 3-edge of
  $H_1$ then $aby$ (and in particular $\{b,y\}$) must be contained in the
  corresponding path in $\mathcal{P}_t$, but $\{b,y\}$ is contained in
  fewer than $t$ paths in $\mathcal{P}_t$ as $m(by)<t$.  Therefore, the
  number of 3-edges of $H_1$ containing $\{a,y\}$ is at most $t-1$, completing 
  the proof of the claim.

We use a greedy algorithm to construct a linear 3-uniform hypergraph $H_2$ from $H_1$ as
follows.  Initially set $V(H_2):=V(G)$, $E(H_2):=\emptyset$ and
$A:=E(H_1)$. In each coming iteration, if $A$ is empty, then stop;
otherwise, choose a 3-edge $\{x,a,y\}\in A$, move it to $E(H_2)$ and then
delete all 3-edges in $A$ which intersect $\{x,a,y\}$ in 2 elements.  When
it ends, we obtain a linear 3-uniform hypergraph $H_2$ such that
\begin{align}
\label{ineq:H2}
|E(H_2)|\ge \frac{|E(H_1)|}{2t}=\frac{|\mathcal{P}_t|}{2t}.
\end{align}

A $r$-uniform hypergraph $H$ is called {\em $r$-partite} if there is a
partition $V(H)=V_1\cup V_2\cup \ldots \cup V_r$ such that for each $e\in
E(H)$ it holds for all $1\le i\le r$ that $|e\cap V_i|=1$.  By randomly and
independently placing each vertex into one of 3 parts, a simple expectation
argument shows that there exists a 3-partite linear 3-uniform hypergraph $H_3$ with
parts $V_1, V_2, V_3$ such that $V(H_3)=V(H_2)$, $E(H_3)\subset E(H_2)$ and
\begin{align}
\label{ineq:H3}
|E(H_3)|\ge \frac{3!}{3^3}|E(H_2)|.
\end{align}
Without loss of generality, we may assume that at least 1/6 of the 3-edges
of $H_3$ have center in $V_2$ and handle in $V_1$.  So there exists a
spanning subhypegraph $H_4$ of $H_3$ satisfying
\begin{align}
\label{ineq:H4}
|E(H_4)|\ge \frac{1}{6}|E(H_3)|
\end{align}
and the property that if $\{v_1,v_2,v_3\}$ is a 3-edge of $H_4$ with
$v_i\in V_i$, then it must be obtained from the critical path $v_1 v_2 v
v_3 \in \mathcal{P}_t$, for some vertex $v$.  By
\eqref{ineq:H2}--\eqref{ineq:H4}, we see
\begin{displaymath}
  |\mathcal{P}_t|\le 54t\cdot |E(H_4)|.
\end{displaymath}

To complete the proof, it suffices to show that $H_4$ has no triangles.
Suppose, on the contrary, that three 3-edges $T_1,T_2,T_3$ of $H_4$ form a
triangle.  Since $H_4$ is linear, we must have $|T_1\cup T_2\cup T_3|=6$.
It also holds for each $1\le i\le 3$ that $|V_i\cap (T_1\cup T_2\cup
T_3)|=2$, as otherwise there is a common vertex in all of 3-edges.  Let
$V_i\cap (T_1\cup T_2\cup T_3)=\{a_i,b_i\}$. Without loss of generality, we
may assume $T_1\cap T_2\cap V_1=\{a_1\}$ such that $T_1=\{a_1,a_2,a_3\}$
and $T_2=\{a_1,b_2,b_3\}$. By symmetry, the only case to consider is
$T_3=\{b_1,a_2,b_3\}$.  Then the construction of $H_4$ ensures that
$a_1a_2\in E(G)$ and there are critical paths $P_{a_1b_3}:=a_1 b_2 u b_3$
and $P_{b_1b_3}:=b_1 a_2 v b_3$ in $\mathcal{P}_t$ for some $u$ and $v$.
Clearly $a_1 a_2 v b_3$ is an $(a_1,b_3)$-path of length 3 distinct from
$P_{a_1b_3}$.  But $P_{a_1b_3}$ should be the unique $(a_1,b_3)$-path of
length at most 3 because $P_{a_1b_3}\in \mathcal{P}$.  This contradiction
finishes the proof.
\end{proof}

\subsection{Covering by critical paths}

The key innovation in our proof for the diameter-3 case is a delicate
accounting of critical paths and edges.  We will construct a family
$\mathcal{F}$ of critical paths such that
\begin{align}
\label{prop:familyF}
\text{every edge in } G \text{ is associated with at least one path in } \mathcal{F}.
\end{align}
This family will be obtained by an iterative greedy algorithm.  In the
$i$-th iteration, we will enlarge $\mathcal{F}$ by adding one or two
critical paths that are selected according to several prescribed rules (see
the algorithm below).  We define the set $P(i)$ to keep track of the
critical paths added in the $i$-th iteration, and for bookkeeping purposes,
we also define sets $P^2(i)$ of some 2-critical paths relevant for the
$i$-th iteration.

During this process, we also maintain an \emph{unsettled set} $U$ which
contains edges of $G$ not ``essentially'' contained in those paths in
$\mathcal{F}$. Let us give the formal definition for $U$.  We have
$|\mathcal{P}_1(e)|\ge 1$ for every $e\in E(G)$, as $G$ is
diameter-3-critical, and thus it is possible that $e$ is associated with
several critical paths of $\mathcal{F}$ added in different iterations.  We
say edge $e$ is {\em settled} in iteration $i$, if $i$ is the first
iteration which adds one critical path associated with the edge $e$.  Note
that given an edge settled in $i$, that edge could be contained in some
critical path added by previous iterations.  Throughout the process, $U$ is
defined to be the up-to-date set consisting of all edges which are not
settled yet.  We define types for edges in $U$ as follows.  An edge $e\in
U$ is of:
\begin{description}
  \item[Type 1:] if there exists a critical path $P\in \mathcal{P}_1(e)$
    which contains at least two associated edges (including $e$) in $U$.

  \item[Type 2:] if it is not of type \1 and there exists a critical path
    $Q\in \mathcal{P}_1(e)$ such that $|Q|=3$ and $e$ is the middle edge of
    $Q$.

  \item[Type 3:] if its type is not in $\{\1, \2\}$ and there exists some
    edge $f\in U$ such that $e$ and $f$ induce a 2-critical path.

  \item[Type 4:] if its type is not in $\{\1, \2, \3\}$ and there exists a
    critical path $R\in \mathcal{P}_1(e)$ such that $|R|=3$ and $e$ is not
    the middle edge of $R$.

  \item[Type 5:] if its type is not in $\{\1, \2, \3, \4\}$ and there
    exists a critical path in $\mathcal{P}_1(e)$ of length two.

  \item[Type 6:] if its type is not in $\{\1, \2, \3, \4, \5\}$.  Note that
    if $e$ is of type \6, then we must have $\mathcal{P}_1(e)=\{e\}$.
\end{description}

We now describe the greedy algorithm. Initially, set
$\mathcal{F}:=\emptyset$ and $U:=E(G)$; we iterate until $U=\emptyset$.  In
the $i$-th iteration, let $t_i$ be the smallest type that the edges in $U$
have.  (Note that when $t_i\ge \2$, $\mathcal{P}_1(e)\cap
\mathcal{P}_1(f)=\emptyset$ for any $e,f\in U$.) We split into cases based
upon the value of $t_i$:
\begin{itemize}
\item [(C1).] If $t_i=\1$, choose a critical path $P_{xy}$ which contains
  at least two associated edges in $U$.  Add $P_{xy}$ to $\mathcal{F}$,
  update $U$ by deleting all associated edges of $P_{xy}$, and then let
  $P(i):=\{P_{xy}\}$ and $P^2(i) := \{P_{xy}-x, P_{xy}-y\}$.

\item [(C2).] If $t_i=\2$, choose an edge $e\in U$ and a critical path
  $P_{xy}\in \mathcal{P}_1(e)$ of length three such that $e$ is the middle
  edge of $P_{xy}$. Add $P_{xy}$ to $\mathcal{F}$, delete $e$ from $U$, and
  let $P(i):=\{P_{xy}\}$ and $P^2(i) := \{P_{xy}-x, P_{xy}-y\}$.

\item [(C3).] If $t_i=\3$:
  \begin{itemize}
    \item [(C3-1).] If there exist $e,f\in U$ such that the path $P:=e\cup
      f$ is in $\mathcal{P}_1(e)$, then choose a critical path $P'\in
      \mathcal{P}_1(f)$, add $P, P'$ to $\mathcal{F}$, delete $e,f$ from
      $U$, and let $P(i):=\{P, P'\}$ and $P^2(i) := \emptyset$. If there
      are multiple choices for $e,f$, choose one that maximizes $|P| +
      |P'|$.

    \item [(C3-2).] Otherwise, choose edges $e,f\in U$ and critical paths
      $P\in \mathcal{P}_1(e)$, $P'\in \mathcal{P}_1(f)$ such that $e\cup f$
      is a 2-critical path. Add $P,P'$ to $\mathcal{F}$, delete $e,f$ from
      $U$, and let $P(i):=\{P, P'\}$ and $P^2(i) := \{e\cup f\}$.
\end{itemize}

\item [(C4).] If $t_i=\4$, choose an edge $e\in U$ and a path $P_{xy}\in
  \mathcal{P}_1(e)$ of length three such that $e$ is incident to $x$ but
  not $y$.  Add $P_{xy}$ to $\mathcal{F}$, delete $e$ from $U$, and let
  $P(i):=\{P_{xy}\}$ and $P^2(i) := \{P_{xy}-y\}$.

\item [(C5).] If $t_i=\5$, choose an edge $e\in U$ and a path $P\in
  \mathcal{P}_1(e)$ of length two. Add $P$ to $\mathcal{F}$, delete $e$
  from $U$, and let $P(i):=\{P\}$ and $P^2(i) := \{e\}$.

\item [(C6).] If $t_i=\6$, then all edges in $U$ are critical paths of
  length one. Choose any $e\in U$, add $e$ to $\mathcal{F}$, delete $e$
  from $U$, and let $P(i):=\{e\}$ and $P^2(i) := \emptyset$.

\end{itemize}

When the algorithm stops, the obtained family $\mathcal{F}$ clearly
satisfies the property \eqref{prop:familyF}.  In each iteration, at least
one edge is deleted from $U$ and thus settled, and one or two critical
paths are added to $\mathcal{F}$.  We also observe that for every edge $e$,
as the algorithm is proceeding, the type of $e$ is nondecreasing (until $e$
is deleted from $U$).  This is because the type of an edge could change
from $\1$ or $\3$ to some larger $k$, but any type other than $\1$ or $\3$
will stay as it is.  This also shows that the sequence $\{t_i\}$ is
nondecreasing.
\begin{lem}
  \label{fact:P(Ii)}
  For every iteration $i$,
  \begin{description}
    \item[(a)] Each path in $P(i)$ is critical, and each path in $P^2(i)$
      is 2-critical.  Each such path contains at least one edge settled in
      the $i$-th iteration.

    \item[(b)] For each path in $P^2(i)$, one of the following holds:
      \begin{itemize}
        \item all edges of $P$ are settled in iteration $i$; or
        \item $P$ is of length two such that one edge of $P$ is settled in
          iteration $i$ and the other is settled in some previous
          iteration.
      \end{itemize}
  \end{description}
\end{lem}

\begin{proof}
  We split into cases based upon which of (C1)--(C6) happened at the $i$-th
  iteration.  First, we observe that (a) and (b) hold trivially if (C3) or
  (C6) occur.  If (C1) or (C2) occurs, we see that the critical path
  $P_{xy}$ must be in the set $\mathcal{P}$ (see its definition in last
  subsection).  Thus (a) and (b) both follow by the fact that all edges of
  $P_{xy}$ are associated with $P_{xy}$.  If (C4) occurs, by definition we
  see that $e\cup e' := P_{xy}-y$ is 2-critical. And we also have $e'\notin
  U$, as otherwise the type of $e$ would be at most $\3$. Hence $e'$ was
  settled before iteration $i$ and the conclusions hold in this case.
  Lastly, if (C5) occurs, it is easy to verify that $e$ indeed is
  2-critical, finishing the proof.
\end{proof}

\begin{lem}
  \label{lem:P(Ii)}
  All sets $P(i)$ and $P^2(i)$ are pairwise disjoint, and for all but at
  most $\frac{n}{2}$ iterations $i$, we have $|P(i) \cup P^2(i)|\ge 2$.
\end{lem}

\begin{proof}
  We first prove the second part.  It is clear from the algorithm that
  $|P(i) \cup P^2(i)|=1$ if and only if $t_i=\6$.  Let us consider the
  first iteration $i$ when all edges in $U$ are of type \6 (hence all
  critical paths of length one).  We claim that this must be a matching.
  Indeed, suppose on the contrary that $e,f\in U$ share a common vertex. It
  is easy to verify that $e\cup f$ actually is 2-critical, implying that
  the type of $e$ is at most $\3$, a contradiction.  Thus, for all but the
  last $|U|\le \frac{n}{2}$ iterations $i$, it holds that $|P(i) \cup
  P^2(i)|\ge 2$.

  It remains to prove the first part, for which it is enough to show that
  $P(i) \cup P^2(i)$ is disjoint from $P(j) \cup P^2(j)$ when $i<j$.  Since
  the sequence $\{t_i\}$ is nondecreasing, it holds that $t_i\le t_j$.
  Consider any $P\in P(i) \cup P^2(i)$ and $Q\in P(j) \cup P^2(j)$, and
  suppose for contradiction that $P=Q$.  First, consider the case when $P
  \in P^2(i)$.  By Lemma \ref{fact:P(Ii)}, every edge of $P$ is settled in
  iteration $i$ or earlier, whereas there is at least one edge of $Q$
  settled in iteration $j$, so $Q \neq P$, contradiction.

  We may therefore assume that $P \in P(i)$. If $Q \in P(j)$, then $P = Q$
  actually contained at least two associated edges which were in $U$ at
  time $i$, and hence we must have been in (C1) at time $i$.  However, then
  we would have deleted all of $P$'s associated edges at that time, making
  it impossible to find $Q$ in $P(j)$ later, contradiction.

  Thus, we may assume that $Q \in P^2(j)$.  By Lemma \ref{fact:P(Ii)}(a),
  at least one edge of $P$ was settled at iteration $i$, and so Lemma
  \ref{fact:P(Ii)}(b) implies that $P=Q$ is of length 2, say $e \cup e'$,
  such that $P\in \mathcal{P}_1(e)$ and $e$ and $e'$ are settled in
  iterations $i$ and $j$ respectively.  From the algorithm, we see that
  (C5) and (C6) do not produce paths of length 2 in $P^2(j)$, and so
  $t_j\le \4$. On the other hand, we have that $t_i\notin \{\1, \2, \4\}$.
  As $t_i\le t_j$, it must be the case that $t_i=\3$ and $t_j=\4$, that is,
  (C3) occurs at time $i$ and (C4) occurs at time $j$.

  Since we must be in (C3) at time $i$, let $f$ be the other settled edge
  in that iteration, and let $P'\in \mathcal{P}_1(f)$ be the other critical
  path in $P(i)$.  Similarly, since we must be in (C4) at time $j$, and
  $e'$ is settled there, let $Q'\in \mathcal{P}_1(e')$ be the corresponding
  critical path from $P(j)$, so that $Q\subsetneq Q'$.  From above, we had
  $e \cup e' = P \in \mathcal{P}_1(e)$, so $e,e'$ form a candidate for
  (C3-1) in iteration $i$ with $|P|=2$ and $|Q'|=3$.  This shows that
  (C3-1) must occur at time $i$, when the two settled edges are $e,f$.
  Looking back on the algorithm, in (C3-1), the two paths in $P(i)$ are $e
  \cup f$, together with a critical path which is associated with either
  $e$ or $f$.  Since $e \cup e' = P \in P(i)$, it must be one of these. 
  We see that $P\neq e \cup f$, because that would force $e' = f$, and $f$ is
  settled at time $i$ while $e'$ is settled at time $j$.  Therefore, $P$
  must be the other critical path in $P(i)$ rather than $e\cup f$. Since $P$
  has length 2, we conclude that the two paths in $P(i)$ both have length
  2.  Yet, as mentioned, $e$ and $e'$ gave rise to an alternate candidate for iteration $i$ in which the
  two paths $P\in \mathcal{P}_1(e)$ and $Q'\in \mathcal{P}_1(e')$ had lengths 2 and 3, and since the algorithm sought
  to maximize the sum of path lengths in $P(i)$ at (C3-1), we have a
  contradiction.
\end{proof}

\subsection{Putting everything together}

The previous section's accounting enables us to complete our proof with
methods similar to to F\"uredi's \cite{Fu} diameter-2 argument.  For a
graph $H$, let $\Di(H)$ denote the set of pairs $\{u,v\}$ in $V(H)$ such
that $u$ and $v$ have disjoint neighborhoods, and let $\di(H):=|\Di(H)|$.
We say that a 2-critical path $P$ is {\em $t$-light} if $|P|=2$ and both
its edges have multiplicity less than $t$.  Recall from the beginning of
Section \ref{sec:diam3} that we define the multiplicity $m(e)$ of an edge $e$
to be the sum $|\mathcal{P}_1(e)| + |\mathcal{P}_2(e)|$, where
$\mathcal{P}_1(e)$ is the set of critical paths associated with $e$ and
$\mathcal{P}_2(e)$ is the set of 2-critical paths containing $e$.  We
record two results from F\"uredi \cite{Fu}, which hold for arbitrary graphs
$H$ (not necessarily diameter-2-critical).

\begin{lem} (Derived from Lemmas 2.1 and 3.3 of F\"uredi \cite{Fu}.)
  \label{lem:t-light}
  For any $n$-vertex graph $H$, $e(H)+\di(H)\le n^2/2$, and the number of
  $t$-light 2-critical paths is less than $27t\cdot \RS(n)$ for any
  positive integer $t$.
\end{lem}

In the rest of the section, let $G$ be a diameter-3-critical graph, and
define $t:=\sqrt{n^2/\RS(n)}$. By the Ruzsa-Szemer\'edi (6,3) Theorem (See Theorem
\ref{thm:RSz}), $t$ tends to infinity as $n\to \infty$.  Let $G_0$ be the
$n$-vertex graph obtained from $G$ by deleting
\begin{itemize}
  \item[(i)] all edges of $G$ whose multiplicity is at least $t$, and
  \item[(ii)] all edges which appear in a $t$-light 2-critical path of $G$.
\end{itemize}
By Lemma \ref{lem:t-light}, at most $54t\cdot \RS(n)$ edges are deleted in
(ii).  To control the number deleted in (i), observe that $\sum m(e) <
3n^2,$ because each critical path is associated with at most 3 edges and
each 2-critical path contains at most 2 edges, and thus it follows from
$\sum |\mathcal{P}_1(e)|\le 3\binom{n}{2}$ and $\sum |\mathcal{P}_2(e)|\le
2\binom{n}{2}$.  So, we delete fewer than $3n^2/t$ edges in (i), producing
\begin{equation}
  \label{ineq:e(G)}
  e(G)\le e(G_0)+\frac{3n^2}{t}+54t\cdot \RS(n),
\end{equation}
and it is clear that, after deleting the edges in (i) and (ii) from $G$,
\begin{align}
\label{prop:destroy2crit}
\text{all 2-critical paths in } G \text{ of length two are destroyed in } G_0.
\end{align}

\begin{lem}
  \label{fact:Di}
  Every critical or 2-critical pair of $G$ is contained in $\Di(G_0)$.
\end{lem}

\begin{proof}
  Consider any 2-critical pair $\{x,y\}$ in $G$. If the 2-critical
  $(x,y)$-path $P$ has length 1, then by definition $\{x,y\}\in \Di(G)$ and
  thus $\{x,y\} \in \Di(G_0)$. If it has length 2, then by
  \eqref{prop:destroy2crit}, at least one edge of $P$ is deleted in $G_0$,
  implying that $\{x,y\}\in \Di(G_0)$.

  Now consider any critical pair $\{x,y\}$. Note that the critical path
  $P_{xy}$ is a shortest $(x,y)$-path. If $|P_{xy}|\le 2$, then $\{x,y\}$
  is also 2-critical. Thus, we may assume that $|P_{xy}|=3$. This shows
  that $N_G(x)\cap N_G(y)=\emptyset$ and thus $\{x,y\}\in \Di(G_0)$.
\end{proof}

Run our algorithm from the previous section on $G$, and let $s$ be the
number of iterations it runs for before it stops.  By Lemma
\ref{fact:P(Ii)}, every path in $P(i) \cup P^2(i)$ is either critical or
2-critical, which also uniquely determines a critical or 2-critical pair of $G$.
Thus, by Lemmas \ref{lem:P(Ii)} and \ref{fact:Di}, we obtain
\begin{equation}
  \label{fact:disj&P(Ii)}
  \di(G_0)\ge \sum_{i} |P(i) \cup P^2(i)|\ge 2s-n/2.
\end{equation}

Let $S(i)$ be the set of edges settled in iteration $i$, which were also
still in $G_0$.  By property \eqref{prop:familyF}, every edge is settled in
some iteration, which shows that $E(G_0)$ is a disjoint union of the sets
$S(i)$.

\begin{lem}
  \label{fact:S(Ii)}
  For all but at most $54t\cdot \RS(n)$ iterations, we have $|S(i)|\le 1$,
  and hence $e(G_0)=\sum |S(i)|\le s+54t\cdot \RS(n)$.
\end{lem}

\begin{proof}
  Consider an arbitrary iteration $i$. If (C2), (C4), (C5) or (C6) occur,
  exactly one edge is settled, giving $|S(i)|\le 1$ trivially.  If (C3)
  occurs, then there are two edges $e$ and $f$ settled such that $e\cup f$
  is a 2-critical path. By \eqref{prop:destroy2crit}, at most one of $e$
  and $f$ remains in $G_0$, and thus it also holds that $|S(i)|\le 1$. It
  remains to consider (C1).  Then there are at least two edges settled in
  this iteration, all of which are associated with a single critical path,
  say $P_{xy}$.  And such $P_{xy}\in \mathcal{P}$ (recall the definitions
  of $\mathcal{P}$ and $\mathcal{P}_t$ at the beginning of Section
  \ref{sec:diam3:Pt}).  If $|P_{xy}|=2$, then such $P_{xy}$ is also
  2-critical, and thus $|S(i)|\le 1$ by the same argument as in (C3).  Only
  $|P_{xy}|=3$ remains. Let $P_{xy}=xaby$.  Note that $xab$ and $aby$ both
  are 2-critical.  By the definition of $G_0$, one can verify that at most
  one edge of $P_{xy}$ can be in $G_0$, unless it is the situation that
  $m(ab)\ge t$, $m(xa)<t$ and $m(by)<t$, that is, $P_{xy}\in
  \mathcal{P}_t$. Thus, $|S(i)|\le 1$ if $P_{xy}\notin \mathcal{P}_t$ and
  $|S(i)| \le 2$ otherwise.  By Lemma \ref{lem:P_t}, we see that
  $|\mathcal{P}_t|\le 54t\cdot \RS(n)$.  This completes the proof.
\end{proof}

We are now ready to prove our final main result, that $G$ has at most
$n^2/6+o(n^2)$ edges.

\begin{proof}[Proof of Theorem \ref{thm:diam3}.]
  Let $G$ be an arbitrary diameter-3-critical graph.  By
  \eqref{fact:disj&P(Ii)} and Lemma \ref{fact:S(Ii)}, we find
  \begin{displaymath}
    \di(G_0)\ge 2e(G_0)-108t\cdot \RS(n)-\frac{n}{2}.
  \end{displaymath}
  Apply Lemma \ref{lem:t-light} to $G_0$, we find
  \begin{displaymath}
    \frac{n^2}{2}
    \ge
    e(G_0)+\di(G_0)
    \ge
    3e(G_0)-108t\cdot \RS(n)-\frac{n}{2},
  \end{displaymath}
  which, together with
  \eqref{ineq:e(G)}, implies that
  \begin{displaymath}
    e(G)
    \le
    \frac{n^2}{6}+\frac{3n^2}{t}+90t\cdot \RS(n)+\frac{n}{6}
    =
    \frac{n^2}{6}+\left(\frac{93n^2}{t}+\frac{n}{6}\right)
    =
    \frac{n^2}{6}+o(n^2),
  \end{displaymath}
  where the equalities follow by the fact that $t=\sqrt{n^2/\RS(n)}\to
  \infty$ as $n\to \infty$.
\end{proof}

\bigskip

\noindent \textbf{Remark.} The proofs in Section 6 actually work for any
diameter-critical graph with diameter at least 3.  The only information
from diameter-3-critical graphs we need in the proof is that for each edge
$e$, the set of 3-critical paths of $e$ is nonempty.  This clearly holds
for all diameter-critical graphs with diameter at least 3. Therefore, for
all such graphs $G$, we have $e(G)\le n^2/6+o(n^2)$ as well.

\end{document}